\documentclass[12pt,amsymb,fullpage]{amsart}
\usepackage{amssymb,amscd,pstricks}
\usepackage{fancyhdr}% fancyhdr related command; details in its documentation
\newtheorem{theorem}{Theorem}[section]
\newtheorem{defn}[theorem]{Definition}

\newtheorem{lemma}[theorem]{Lemma}

\newtheorem{eple}[theorem]{Example}
\newtheorem{rmk}[theorem]{Remarks}
\newtheorem{dsc}[theorem]{Discussion}
\newtheorem{nota}[theorem]{Notation}

\newsavebox{\indbin}
\savebox{\indbin}{\begin{picture}(0,0)
\newlength{\gnu}
\settowidth{\gnu}{$\smile$} \setlength{\unitlength}{.5\gnu}
\put(-1,-.65){$\smile$} \put(-.25,.1){$|$}
\end{picture}}

\newcommand{\be}{\begin{enumerate}}
\newcommand{\bd}{\begin{defn}}
\newcommand{\bt}{\begin{theorem}}
\newcommand{\bl}{\begin{lemma}}
\newcommand{\ee}{\end{enumerate}}
\newcommand{\ed}{\end{defn}}
\newcommand{\et}{\end{theorem}}
\newcommand{\el}{\end{lemma}}

\begin{document}
\title{A Simple Proof of the Uniform
Convergence of
Fourier Series in Solutions to the Wave Equation}
\author{Tristram de Piro}
\maketitle
\begin{abstract}
Using methods of \cite{dep1}, we show that the time dependent Fourier series of any $F\in C^{\infty}(0,L)$, solving the wave equation, with $F(0,t)=F(L,t)=0$, converges uniformly to $F$, on $[0,L]$, and find an explicit formula for such series.
\end{abstract}
\begin{defn}
\label{wave}
We let;\\

$C^{n}([0,L])=\ \ \ \{f\in C([0,L]): f|_{(0,L)}\in C^{n}(0,L),$\\

\indent \ \ \ \ \ \ \ \ \ \ \ \ \ \ \ \ \ \ $(\forall i\leq n)\exists r_{i}\in C[0,L], r_{i}|_{(0,L)}=f^{(i)}\}$, (\footnote{This definition is equivalent to, $(\forall i\leq n)\{f^{(i)}_{+}(0), f^{(i)}_{-}(L)\}$ exist, where, for $i\leq n$, $f^{(i)}_{+}(0)$ is defined inductively, by ,$f^{(i)}_{+}(0)=lim_{s\rightarrow 0}{f^{(i-1)}(s)-f^{(i-1)}_{+}(0)\over s}$, and, similarly, for $f^{(i)}_{-}(L)$. In order to see this, just observe that, for $i\leq n$, $lim_{s\rightarrow 0}{f^{(i-1)}(s)-f^{(i-1)}_{+}(0)\over s}=lim_{s\rightarrow 0}f^{(i)}(s)$, by L'Hopital's Rule and the Intermediate Value Theorem.})\\

$C^{n}_{0}([0,L])=\ \ \ \{f\in C^{n}([0,L]):\ f(0)=f(L)=0\}\\$\\

$C^{\infty}([0,L])=\ \ \ \{f\in C([0,L]): f|_{(0,L)}\in C^{\infty}(0,L)$\\

\indent \ \ \ \ \ \ \ \ \ \ \ \ \ \ \ \ \ \ \ \ $\forall({i}\leq n)\exists r_{i}\in C([0,L]), r_{i}|_{(0,L)}=f^{(i)}$\}\\

$C^{\infty}_{0}([0,L])=\ \ \ \{f\in C^{\infty}([0,L]):\ f(0)=f(L)=0\}$\\

$C^{n}([0,L]\times\mathcal{R})=\{F\in C([0,L]\times\mathcal{R}): F|_{(0,L)\times\mathcal{R}}\in C^{n}((0,L)\times$\\

\indent $
 \ \ \ \ \ \ \ \ \ \ \ \ \ \ \ \ \ \ \ \ \ \ \ \mathcal{R}), (\forall i\leq n)\exists r_{i}\in C([0,L]\times\mathcal{R})\ r_{i}|_{(0,L)\times\mathcal{R}}={\partial^{i} F\over \partial x^{i}}\}$\\

 $C^{n}_{0}([0,L]\times\mathcal{R})=\{F\in C^{n}([0,L]\times\mathcal{R}):\ (\forall t\in\mathcal{R}), F(0,t)$\\

 \indent \ \ \ \ \ \ \ \ \ \ \ \ \ \ \ \ \ \ \ $=F(L,t)=0\}$\\

$C^{\infty}([0,L]\times\mathcal{R})=\{F\in C([0,L]\times\mathcal{R}): F|_{(0,L)\times\mathcal{R}}\in C^{\infty}((0,L)\times$\\

\indent $
 \ \ \ \ \ \ \ \ \ \ \ \ \ \ \ \ \ \ \ \ \ \ \ \mathcal{R}), \forall(i\leq n)\exists r_{i}\in C([0,L]\times\mathcal{R})\ r_{i}|_{(0,L)\times\mathcal{R}}={\partial^{i} F\over \partial x^{n}}\}$\\

$C^{\infty}_{0}([0,L]\times\mathcal{R})=\{F\in C^{\infty}([0,L]\times\mathcal{R}):\ (\forall t\in\mathcal{R}), F(0,t)=$\\

\indent \ \ \ \ \ \ \ \ \ \ \ \ \ \ \ \ \ \  \ \ \ \ \ \ $F(L,t)=0\}$\\

We let $\{T,M,L\}$ denote the tension,mass and length of a string, with $\mu={M/L}$, the mass per unit length. The wave equation;\\

${\partial^{2} F\over \partial t^{2}}={T\over\mu}{\partial^{2} F\over \partial x^{2}}$ $(*)$\\

with boundary condition $F(0,t)=F(L,t)=0$, for $t\in\mathcal{R}$, describes the motion of a vibrating string under tension, fixed at the endpoints, (\footnote{By a solution to the wave equation, we mean $F\in C^{\infty}_{0}([0,L]\times\mathcal{R})$, satisfying the equation $(*)$ on $(0,L)\times\mathcal{R}$}).\\

We say that $h\in C([-L,L])$ is symmetric, if $h(-x)=h(x)$, for $x\in [-L,L]$, (with endpoints identified). We say that $h\in C([-L,L])$ is asymmetric if $h(-x)=-h(x)$,for $x\in [-L,L]$, (with endpoints identified). We use the same notation as above for functions on $[-L,L]$, (with endpoints identified). We define;\\

$C^{n}((-L,0)\cup (0,L))=\{f\in C((-L,0)\cup (0,L)): \exists(r_{1}\in C^{n}([-L,0]),r_{2})\in C^{n}([0,L]), r_{1}|_{(-L,0)}=f|_{(-L,0)}, r_{2}|_{(0,L)}=f|_{(0,L)} \}$\\

\end{defn}

We require the following results;\\

\begin{lemma}
\label{symasym2}
Let $h\in C([-L,L])$ be asymmetric, with $h|_{(-L,0)\cup (0,L)}\in C^{1}((-L,0)\cup (0,L))$, $(*)$, then $h(0)=h(L)=h(-L)=0$, $h'_{+}(-L)=h'_{-}(L)$, $h'_{+}(0)=h'_{-}(0)$, $h'\in C([-L,L])$, and $h'$ is symmetric. Let $h\in C([-L,L])$ be symmetric, with $h|_{(-L,0)\cup (0,L)}\in C^{1}((-L,0)\cup (0,L))$, $(**)$, and $h'_{+}(-L)=h'_{-}(L)=0$, $h'_{+}(0)=h'_{-}(0)=0$, then $h'\in C([-L,L])$ is asymmetric.

\end{lemma}

\begin{proof}
For the first part, we have, if $h$ is asymmetric, satisfying $(*)$, then $h(L)=-h(-L)=-h(L)$ and $h(0)=-h(-0)=-h(0)$, so $h(0)=h(L)=h(-L)=0$. We have that;\\

$h'_{+}(-L)=lim_{s\rightarrow 0}{h(-L+s)-h(-L)\over s}$ \\

$=lim_{s\rightarrow 0} {h(-L+s)\over s}=lim_{s\rightarrow 0}{-h(L-s)\over s}$ (by asymmetry and $h(-L)=0$)\\

$=lim_{s\rightarrow 0}{h(L)-h(L-s)\over s}=h'_{-}(L)$ (as $h(L)$=0)\\

Similarly, $h'_{+}(0)=h'_{-}(0)$. By L'Hopital's rule, and the fact that $h|_{(-L,0)\cup (0,L)}\in C^{1}((-L,0)\cup (0,L))$, we have that $lim_{s\rightarrow 0}h'(s)=lim_{s\rightarrow 0}{h(s)-h(0)\over s}$\\
$=h'_{+}(0)$, and, similarly, $lim_{s\rightarrow 0}h'(-s)=h'_{-}(0)$, $lim_{s\rightarrow 0}h'(L-s)=h'_{-}(L)$, $lim_{s\rightarrow 0}h'(-L+s)=h'_{-}(L)$. Hence, $h'\in C([-L,L])$, and $h'$ is symmetric by the fact that $h(x)=-h(-x)$, and, therefore, $h'(x)=h'(-x)$, for $x\in (-L,0)\cup (0,L)$, and, automatically, $h'(L)=h'(-L)$, $h'(0)=h'(-0)$, as these points are fixed.\\

Let $h\in C([-L,L])$ be symmetric, satisfying $(**)$. By L'Hopital's rule, and the fact that $h|_{(-L,0)\cup (0,L)}\in C^{1}((-L,0)\cup (0,L))$, we have that $lim_{s\rightarrow 0}h'(s)=lim_{s\rightarrow 0}{h(s)-h(0)\over s}=h'_{+}(0)=0=h'_{-}(0)=lim_{s\rightarrow 0}{h(0)-h(-s)\over s}=lim_{s\rightarrow 0}h'(-s)$ and, similarly, $lim_{s\rightarrow 0}h'(L-s)=h'_{-}(L)$, $lim_{s\rightarrow 0}h'(-L+s)=h'_{-}(L)$. Hence, $h'\in C([-L,L])$, and $h'$ is symmetric by the fact that $h(x)=h(-x)$, and, therefore, $h'(x)=-h'(-x)$, for $x\in (-L,0)\cup (0,L)$, and, automatically, $h'(L)=h'(-L)=0$, $h'(0)=h'(-0)=0$, as these points are fixed.\\

\end{proof}

\begin{lemma}
\label{asymmetric}
Let $f\in C^{2}([0,L])$, such that $f(0)=f(L)=0$ and $f_{+}''(0)=f_{+}''(L)=0$, $(*)$,  then there exists $h\in C^{2}([-L,L])$, (with endpoints identified), such that $h|_{[0,L]}=f$, $h$ is asymmetric about $0$, and $h'$ is symmetric about $0$. Let $f\in C^{2}([0,L])$, such that $f(0)=f(L)=0$ and $f_{+}'(0)=f_{-}'(L)=0$, $(**)$, then there exists $h\in C^{2}([-L,L])$, (with endpoints identified), such that $h|_{[0,L]}=f$, $h$ is symmetric about $0$, and $h'$ is asymmetric about $0$.

\end{lemma}

\begin{proof}
Suppose that $f$ satisfies $(*)$ and let $h(x)=f(x)$, for $x\in [0,L]$, and $h(x)=-f(-x)$, for $x\in [-L,0)$. Then clearly $h$ is asymmetric about $0$, $h(0)=h(L)=h(-L)=0$, and $h\in C([-L,L])$. Moreover, $h|_{(-L,0)\cup (0,L)}\in C^{1}((-L,0)\cup (0,L))$, as $f\in C^{1}([0,L])$. By Lemma \ref{symasym2}, we have that $h'\in C([-L,L])$, and $h'$ is symmetric. Moreover, $h'|_{(-L,0)\cup (0,L)}\in C^{1}((-L,0)\cup (0,L))$, as $f\in C^{2}([0,L])$ and $f'\in C^{1}([0,L])$. We have that;\\

$(h')'_{+}(-L)=lim_{s\rightarrow 0^{+}}{h'(-L+s)-h'(-L)\over s}$\\

$=lim_{s\rightarrow 0}{h'(L-s)-h'(L)\over s}$ (by asymmetry)\\

$=-lim_{s\rightarrow 0}{h'(L)-h'(L-s)\over s}$\\

$=-lim_{s\rightarrow 0}{f'(L)-f'(L-s)\over s}=-f_{+}''(L)=0$\\

and;\\

$(h')'_{-}(L)=lim_{s\rightarrow 0^{+}}{h'(L)-h'(L-s)\over s}$\\

$=lim_{s\rightarrow 0}{h'_{-}(L)-f'(L-s)\over s}$\\

$=lim_{s\rightarrow 0}{f'_{-}(L)-f'(L-s)\over s}$\\

$=lim_{s\rightarrow 0}{f'(L)-f'(L-s)\over s}=f_{+}''(L)=0$\\

Similarly, $(h')'_{+}(0)=f_{+}''(0)=0$, $(h')'_{-}(0)=-f_{+}''(0)=0$\\

Applying Lemma \ref{symasym2} again, we obtain that $(h')'\in C[-L,L]$, hence $h\in C^{2}([-L,L])$.\\

Suppose that $f$ satisfies $(**)$ and let $h(x)=f(x)$, for $x\in [0,L]$, $h(x)=f(-x)$, for $x\in [-L,0)$. Then $h$ is symmetric and $h|_{(-L,0)\cup (0,L)}\in C^{1}((-L,0)\cup (0,L))$. Moreover;\\

$h'_{+}(-L)=lim_{s \rightarrow 0}{h(-L+s)-h(-L)\over s}$\\

$=lim_{s\rightarrow 0}{h(L-s)-h(L)\over s}$\\

$=lim_{s\rightarrow 0}{f(L-s)-f(L)\over s}$\\

$=-lim_{s\rightarrow 0}{f(L)-f(L-s)\over s}$\\

$=-f'_{-}(L)=0$\\

Similarly, $h'_{-}(L)=f'_{-}(L)=0$, $h'_{+}(0)=f'_{+}(0)=0$, and $h'_{-}(0)=-f'_{+}(0)=0$. Again, applying Lemma \ref{symasym2}, we obtain that $h'\in C([-L,L])$ is asymmetric. We have that $h'|_{(-L,0)\cup (0,L)}\in C^{1}((-L,0)\cup (0,L))$, as $f\in C^{2}([0,L])$ and $f'\in C^{1}([0,L])$. Applying Lemma \ref{symasym2}, we obtain that $(h')'\in C[-L,L]$, hence $h\in C^{2}([-L,L])$.\\
\end{proof}

\begin{lemma}
\label{asymmetric'}
Let $f\in C^{4}([0,L])$, such that $f(0)=f(L)=0$ and $f_{+}^{(2)}(0)=f_{+}^{(2)}(L)=0$, $f_{+}^{(4)}(0)=f_{+}^{(4)}(L)=0$, $(*)$,  then there exists $h\in C^{4}([-L,L])$, (with endpoints identified), such that $h|_{[0,L]}=f$, $\{h,h^{(2)}\}$ are asymmetric about $0$, and $\{h^{(1)},h^{(3)}\}$ are symmetric about $0$. Let $f\in C^{4}([0,L])$, such that $f(0)=f(L)=0$ and $f_{+}^{(1)}(0)=f_{+}^{(1)}(L)=0$, $f_{+}^{(3)}(0)=f_{+}^{(3)}(L)=0$, $(**)$, then there exists $h\in C^{4}([-L,L])$, (with endpoints identified), such that $h|_{[0,L]}=f$, $\{h,h^{(2)}\}$ are symmetric about $0$, and $\{h^{(1)},h^{(3)}\}$ are asymmetric about $0$.
\end{lemma}

\begin{proof}
For the first part, let $h$ be defined as in \ref{asymmetric}, then $h\in C^{2}([-L,L])$, (with endpoints identified), $h|_{[0,L]}=f$, $h$ is asymmetric about $0$ and $h^{(1)}$ is symmetric about $0$. We have that $f^{(2)}\in C^{2}([0,L])$, $f_{+}^{(2)}(0)=f_{+}^{(2)}(L)=0$, and $f_{+}^{(4)}(0)=f_{+}^{(4)}(L)=0$, so $f^{(2)}$ satisfies the hypotheses of Lemma \ref{asymmetric}. Moreover, $h^{(2)}(x)=f^{(2)}(x)$, for $x\in [0,L]$, and $h^{(2)}(-x)=-f^{(2)}(-x)$, for $x\in [-L,0)$. Then, by the result of \ref{asymmetric}, we have that $h^{(2)}\in C^{2}([-L,L])$, (with endpoints identified), $h^{(2)}$ is asymmetric about $0$ and $h^{(3)}$ is symmetric about $0$. Hence $h\in C^{4}([-L,L])$, and the remaining claims are clear. The proof of the second part of the lemma follows the same strategy.

\end{proof}

\begin{lemma}
\label{decomp}
Let $f\in C^{\infty}_{0}([0,L])$,  then there exists $\{f_{1},f_{2}\}\subset C^{\infty}_{0}([0,L])$, with $f'_{1,+}(0)=f'_{1,+}(L)=0$, $f''_{2,+}(0)=f''_{2,+}(L)=0$, such that $f=f_{1}+f_{2}$.

\end{lemma}

\begin{proof}
Consider the equations $g(0)=g(L)=0$, $g'(0)=g'(L)=0$, $g''(0)=f''_{+}(0)$ and $g''(L)=f''_{-}(L)$, $(*)$ on the space $V_{6}=\{g\in \mathcal{R}[x]:deg(g)=5\}$. Let $T:V_{6}\rightarrow\mathcal{R}^{6}$ be given by;\\

$T(g)=(g(0),g(L),g'(0),g'(L),g''(0),g''(L))$\\

We have that $Ker(T)=0$, as if $T(g)=0$, then, clearly $g(x)=dx^{3}+ex^{4}+fx^{5}$, with $\{d,e,f\}\subset\mathcal{R}$, then, $g'(x)=3dx^{2}+4ex^{3}+5fx^{4}$, $g''(x)=6dx+12ex^{2}+20fx^{3}$, and we have that $g(L)=g'(L)=g''(L)=0$, iff;\\

$A \centerdot \begin{pmatrix} d\\ e\\ f\\ \end{pmatrix} = \begin{pmatrix} 0\\ 0\\ 0\\ \end{pmatrix}, A=\begin{pmatrix} 1 & L & L^{2}\\ 3 & 4L & 5L^{2}\\ 6 & 12L & 20L^{2}\\ \end{pmatrix}$\\

We have that $det(A)=2L^{3}\neq 0$, hence, $d=e=f=0$, as required. Then, $T$ is onto, by the rank-nullity theorem, hence, we can find a solution to $(*)$, corresponding to $T(g)=v_{1}$, where $v_{1}=(0,0,0,0,f''(0),f''(L))$. Let $f_{1}$ be the unique polynomial in $V_{5}$, satisfying these conditions, and let $f_{2}=f-f_{1}$. It is now a simple calculation to see that $\{f_{1},f_{2}\}$ satisfy the required conditions.
\end{proof}

\begin{lemma}
\label{decomp'}
Let $f\in C^{\infty}_{0}([0,L])$, and $n\in\mathcal{Z}_{\geq 1}$, then there exists $\{f_{1},f_{2}\}\subset C^{\infty}_{0}([0,L])$, with $f^{(2j-1)}_{1,+}(0)=f^{(2j-1)}_{1,-}(L)=0$, $f^{(2j)}_{1,+}(0)=f^{(2j)}_{1,-}(L)=0$, for $1\leq j\leq n$, such that $f=f_{1}+f_{2}$.
\end{lemma}

\begin{proof}
Consider the equations $g(0)=g(L)=0$, $g^{(2j-1)}(0)=g^{(2j-1)}(L)=0$, and $g^{(2j)}(0)=f^{(2j)}_{+}(0)$, $g^{(2j)}(L)=f^{(2j)}_{-}(L)$, for $1\leq j\leq n$, $(*)$, on the space $V_{2(2n+1)}=\{g\in \mathcal{R}[x]:deg(g)=4n+1\}$. Let $T:V_{2(2n+1)}\rightarrow\mathcal{R}^{2(2n+1)}$ be given by;\\

$(T(g))_{1}=g(0)$\\

$(T(g))_{2}=g(L)$\\

$(T(g))_{1+2j}=g^{(j)}(0)$\\

$(T(g))_{2+2j}=g^{((j))}(L)$ ($1\leq j\leq 2n$)\\

We have that $Ker(T)=0$, as if $T(g)=0$, then, using the fact that $g(0)=0$, $g^{(j)}(0)=0$, for $1\leq j\leq 2n$, we have  $g(x)=\sum_{i=2n+1}^{4n+1}a_{i}x^{i}$, with $a_{i}\in\mathcal{R}$, for $2n+1\leq i\leq 4n+1$. Then, for $1\leq j\leq 2n$;\\

$g^{(j)}(x)=\sum_{i=2n+1}^{4n+1}{i!\over (i-j)!}a_{i}x^{i-j}$\\

and we have that $g(L)=0$, $g^{(j)}(L)=0$, for $1\leq j\leq 2n$ iff;\\

$A \centerdot \begin{pmatrix} a_{2n+1}\\ \centerdot\\ \centerdot\\ a_{2n+i}\\ \centerdot\\ \centerdot\\ a_{4n+1}\\  \end{pmatrix} = \begin{pmatrix} 0\\ \centerdot\\ \centerdot\\ 0\\ \centerdot\\ \centerdot\\ 0\\ \end{pmatrix}, A=\begin{pmatrix} 1 & \centerdot & \centerdot & L^{i-1} &  \centerdot & \centerdot & L^{2n-1}\\ \centerdot & \centerdot & \centerdot & \centerdot & \centerdot & \centerdot & \centerdot\\ \centerdot & \centerdot & \centerdot & \centerdot & \centerdot & \centerdot & \centerdot\\ {(2n+1)!\over (2n+2-j)!} & \centerdot & \centerdot & {(2n+i)!L^{i-1}\over(2n+1+i-j)!} &  \centerdot & \centerdot & {(4n+1)!L^{2n-1}\over(4n+1-j)!}\\ \centerdot & \centerdot & \centerdot & \centerdot & \centerdot & \centerdot & \centerdot\\ \centerdot & \centerdot & \centerdot & \centerdot & \centerdot & \centerdot & \centerdot\\ {(2n+1)!\over 2!} & \centerdot & \centerdot & {(2n+i)!L^{i-1}\over(i+1)!} &  \centerdot & \centerdot & {(4n)!L^{2n-1}\over(2n+1)!}\\ \end{pmatrix}$\\

for $1\leq i,j\leq 2n$.\\

We have that $det(A)=cL^{n(2n-1)}\neq 0$, (work out $c$) hence, $a_{i}=0$, for $2n+1\leq i\leq 4n+1$, as required. Then, $T$ is onto, by the rank-nullity theorem, hence, we can find a solution to $(*)$, corresponding to $T(g)=v_{1}$, where;\\

 $(v_{1})_{j}=0$, $1\leq j\leq 2$\\

 $(v_{1})_{j}=0$, ($j=4k-1$, $j=4k$, $1\leq k\leq n$)\\

 $(v_{1})_{j}=f_{-}^{(2k)}(0)$, ($j=4k+1$, $1\leq k\leq n$)\\

 $(v_{1})_{j}=f_{-}^{(2k)}(L)$, ($j=4k+2$, $1\leq k\leq n$)\\

Let $f_{1}$ be the unique polynomial in $V_{2(2n+1)}$, satisfying these conditions, and let $f_{2}=f-f_{1}$. It is now a simple calculation to see that $\{f_{1},f_{2}\}$ satisfy the required conditions.
\end{proof}

\begin{lemma}
\label{extension}
Let $f\in C^{\infty}_{0}([0,L])$, then, for all $\epsilon>0$, there exists $g\in C^{2}([-L,L])$, such that;\\

$g|_{[\epsilon,L-\epsilon)}=f|_{[\epsilon,L-\epsilon)}$.\\

\end{lemma}

\begin{proof}
By Lemma \ref{decomp}, we can find $\{f_{1},f_{2}\}\subset C^{\infty}_{0}([0,L])$, with $f'_{1,+}(0)=f'_{1,+}(L)=0$, $f''_{2,+}(0)=f''_{2,+}(L)=0$, such that $f=f_{1}+f_{2}$. By Lemma \ref{asymmetric}, we can find $\{g_{1},g_{2}\}\subset C^{2}_{0}([0,L])$, with $g_{1}|_{[0,L]}=f_{1}$, $g_{2}|_{[0,L]}=f_{2}$ and $g_{1}$ symmetric, $g_{2}$ asymmetric. Let $g=g_{1}+g_{2}$, then $g\in C^{2}_{0}([0,L])$, and $g|_{[\epsilon,L-\epsilon)}=f|_{[\epsilon,L-\epsilon)}$.
\end{proof}

\begin{lemma}
\label{extension'}
Let $f\in C^{\infty}_{0}([0,L])$, then, for all $\epsilon>0$, there exists $g\in C^{4}([-L,L])$, such that;\\

$g|_{[\epsilon,L-\epsilon)}=f|_{[\epsilon,L-\epsilon)}$.\\

\end{lemma}

\begin{proof}
By Lemma \ref{decomp'}, we can find $\{f_{1},f_{2}\}\subset C^{\infty}_{0}([0,L])$, with $f^{(1)}_{1,+}(0)=f^{(1)}_{1,-}(L)=0$, $f^{(3)}_{1,+}(0)=f^{(3)}_{1,-}(L)=0$,  $f^{(2)}_{2,+}(0)=f^{(2)}_{2,-}(L)=0$, $f^{(4)}_{2,+}(0)=f^{(4)}_{2,-}(L)=0$ such that $f=f_{1}+f_{2}$. By Lemma \ref{asymmetric'}, we can find $\{g_{1},g_{2}\}\subset C^{4}_{0}([-L,L])$, with $g_{1}|_{[0,L]}=f_{1}$, $g_{2}|_{[0,L]}=f_{2}$ and $g_{1}$ symmetric, $g_{2}$ asymmetric. Let $g=g_{1}+g_{2}$, then $g\in C^{4}_{0}([-L,L])$, and $g|_{[\epsilon,L-\epsilon)}=f|_{[\epsilon,L-\epsilon)}$.
\end{proof}

\begin{lemma}
\label{asymmetric2}
Let $F\in C^{2}([0,L]\times\mathcal{R})$, such that $F(0,t)=F(L,t)=0$, for all $t\in\mathcal{R}$, and let $F_{t,+}''(0)=F_{t,+}''(L)=0$, $(*)$,  then there exists $H\in C^{2}([-L,L]\times\mathcal{R})$, (with endfaces identified), such that $H|_{[0,L]\times\mathcal{R}}=F$, $H$ is asymmetric about $0$, and ${\partial H\over\partial x}$ is symmetric about $0$. Let $F\in C^{2}([0,L]\times\mathcal{R})$, such that $F(0)=F(L)=0$ and $F_{t,+}'(0)=f_{t,-}'(L)=0$, $(**)$, then there exists $H\in C^{2}([-L,L]\times\mathcal{R})$, (with endfaces identified), such that $H|_{[0,L]\times\mathcal{R}}=F$, $H$ is symmetric about $0$, and ${\partial H\over\partial x}$ is asymmetric about $0$.
\end{lemma}
\begin{proof}
Suppose that $F$ satisfies $(*)$ and let $H(x,t)=F(x,t)$, for $(x,t)\in [0,L]\times\mathcal{R}$, and $H(x,t)=-F(-x,t)$, for $(x,t)\in [-L,0)\times\mathcal{R}$, $(***)$. Using the result of Lemma \ref{asymmetric}, we have, for $t\in\mathcal{R}$, that $H_{t}\in C^{2}([-L,L])$, $(****)$, $H_{t}|_{[0,L]}=F_{t}$, $(*******)$, $H_{t}$ is asymmetric about $0$, $(\dag)$, and $H'_{t}$ is symmetric about $0$, $(\dag\dag)$. Let $r_{2}\in C([0,L]\times\mathcal{R})$ be given, as in Definition \ref{wave}, for $F$, so that $r_{2}|_{(0,L)\times\mathcal{R}}={\partial^{2}H\over\partial x^{2}}|_{(0,L)\times\mathcal{R}}$, $(******)$, and let $r_{2,l}\in C([-L,0]\times\mathcal{R})$ be given by $r_{2,l}(x,t)=-r_{2}(-x,t)$, for $(x,t)\in [-L,0]\times\mathcal{R}$, so that $r_{2,l}|_{(-L,0)\times\mathcal{R}}={\partial^{2}H\over\partial x^{2}}|_{(-L,0)\times\mathcal{R}}$, $(*****)$. Let $R_{2}$ be defined by $R_{2}(x,t)=r_{2}(x,t)$, if $(x,t)\in [0,L]\times\mathcal{R}$, and $R_{2}(x,t)=r_{2,l}(x,t)$, if $(x,t)\in [-L,0]\times\mathcal{R}$. Then $R_{2,t}|_{[-L,L]}=H_{t}$, hence, by $(****)$, in fact, $R_{2}\in C([-L,L]\times\mathcal{R})$, and, by $(***)$, $(*****)$, $(******)$, $R_{2}|_{((-L,0)\cup (0,L))\times\mathcal{R}}={\partial^{2}H\over\partial x^{2}}$. It follows that $H\in C^{2}([-L,L]\times\mathcal{R})$ (with endpoints identified). By $(*******)$, we obtain immediately that $H|_{[0,L]\times\mathcal{R}}=F$. The fact that $H$ is asymmetric about $0$, is obvious, from $(\dag)$. In order to see the final claim, let $r_{1}\in C([0,L]\times\mathcal{R})$ be given, as above, $r_{1,l}\in C([-L,0]\times\mathcal{R})$, be given by, $r_{1,l}(x,t)=r_{1}(-x,t)$, and $R_{1}\in C([0,L]\times\mathcal{R})$ be defined by $R_{1}(x,t)=r_{1}(x,t)$, if $(x,t)\in [0,L]\times\mathcal{R}$, and $R_{1}(x,t)=r_{1,l}(x,t)$, if $(x,t)\in [-L,0]\times\mathcal{R}$. It is easy to see, as above, that $R_{1}\in C([-L,L]\times\mathcal{R})$ and $R_{1}|_{(-L,L)\times \mathcal{R}}={\partial H \over\partial x}$. Then, for $t\in\mathcal{R}$, $R_{1,t}|_{(-L,L)}=(H_{t})'$, so that, for $t\in\mathcal{R}$, $R_{1,t}=r_{1,t}$, $(\dag\dag\dag)$, where $r_{1,t}$ is given, as in Definition \ref{wave}, for each $H_{t}$.  Then, the fact that ${\partial H\over \partial x}$ is symmetric about $0$, follows from the pointwise property $(\dag\dag)$, and, $(\dag\dag\dag)$. The second part of the lemma is the similar, following the proof above
\end{proof}

\begin{lemma}
\label{asymmetric2'}
Let $F\in C^{4}([0,L]\times\mathcal{R})$, such that $F(0,t)=F(L,t)=0$, for all $t\in\mathcal{R}$, and let $F_{t,+}^{(2)}(0)=F_{t,+}^{(2)}(L)=0$, $F_{t,+}^{(4)}(0)=F_{t,+}^{(4)}(L)=0$  $(*)$, then there exists $H\in C^{4}([-L,L]\times\mathcal{R})$, (with endfaces identified), such that $H|_{[0,L]\times\mathcal{R}}=F$, $H,{\partial^{2} H\over\partial x^{2}}$ are asymmetric about $0$, and ${\partial H\over\partial x},{\partial^{3} H\over\partial x^{3}}$ are symmetric about $0$. Let $F\in C^{4}([0,L]\times\mathcal{R})$, such that $F(0)=F(L)=0$ and $F_{t,+}^{(1)}(0)=F_{t,-}^{(1)}(L)=0$, $F_{t,+}^{(3)}(0)=F_{t,-}^{(3)}(L)=0$  $(**)$, then there exists $H\in C^{4}([-L,L]\times\mathcal{R})$, (with endfaces identified), such that $H|_{[0,L]\times\mathcal{R}}=F$, $H,{\partial^{2} H\over\partial x^{2}}$ are symmetric about $0$, and ${\partial H\over\partial x},{\partial^{3} H\over\partial x^{3}}$ are asymmetric about $0$.
\end{lemma}
\begin{proof}
For the first part, by Lemma \ref{asymmetric2}, we can find $H\in C^{2}([-L,L]\times\mathcal{R})$, with $H|_{[0,L]\times\mathcal{R}}=F$, $H$ asymmetric about $0$, and ${\partial H\over \partial x}$ symmetric about $0$. We have that ${\partial^{2}F\over\partial x^{2}}$ satisfies the conditions of Lemma \ref{asymmetric2}, as, by the assumptions,  ${\partial^{2}F\over\partial x^{2}}\in C^{2}([0,L]\times\mathcal{R})$, ${\partial^{2}F\over\partial x^{2}}(0,t)={\partial^{2}F\over\partial x^{2}}(L,t)=0$ and $({\partial^{2}F\over\partial x^{2}})_{t,+}^{(2)}(0)=F_{t,+}^{(4)}(0)=F_{t,-}^{(4)}(L)=({\partial^{2}F\over\partial x^{2}})_{t,-}^{(2)}(L)=0$, for all $t\in\mathcal{R}$, (\footnote{Here, we use the fact that, for $t\in\mathcal{R}$, $(({\partial^{2}F\over\partial x^{2}})_{t})|_{(0,L)}=(F_{t})^{(2)}|_{(0,L)}$, so $(({\partial^{2}F\over\partial x^{2}})_{t})^{(2)}|_{(0,L)}=(F_{t})^{(4)}|_{(0,L)}$, $(*)$, and, using Definition \ref{waves}, the limits $({\partial^{2}F\over\partial x^{2}})_{t,+}^{(2)}(0)=F_{t,+}^{(4)}(0)$ are recovered uniquely from the relation $(*)$      }). Moreover, by definition of $H$, we have that ${\partial^{2}H\over\partial x^{2}}(x',t)={\partial^{2}F\over\partial x^{2}}(x',t)$, for $(x',t)\in ([0,L]\times\mathcal{R})$, and  ${\partial^{2}H\over\partial x^{2}}(x',t)=-{\partial^{2}F\over\partial x^{2}}(-x',t)$, for $(x',t)\in ((-L,0)\times\mathcal{R})$. Hence, by the conclusion of Lemma \ref{asymmetric2}, we have that ${\partial^{2}H\over\partial x^{2}}\in C^{2}([-L,L])$, (with endfaces identified) ${\partial^{2}H\over\partial x^{2}}$ is symmetric about $0$, and ${\partial^{3}H\over\partial x^{3}}$ is asymmetric about $0$, as required.

\end{proof}

\begin{lemma}
\label{decomp2}
Let $F\in C^{\infty}_{0}([0,L]\times\mathcal{R})$,  then there exists $\{F_{1},F_{2}\}\subset C^{\infty}_{0}([0,L]\times\mathcal{R})$, with $F'_{1,t,+}(0)=F'_{1,t,-}(L)=0$, $F''_{2,t,+}(0)=F''_{2,t,-}(L)=0$, such that $F=F_{1}+F_{2}$.

\end{lemma}

\begin{proof}
This is just a uniform version of Lemma \ref{decomp}. Let;\\

$v_{1,t}=(0,0,0,0,F''_{t,+}(0),F''_{t,+}(L))$, $p_{1,t}=T^{-1}(v_{1,t})$\\

 Then;\\

$p_{1,t}=\sum_{i=0}^{5}d_{i}(t)x^{i}$\\

where the coefficients $d_{i}(t)=\lambda_{i}F''_{t,+}(0)+\mu_{i}F''_{t,+}(L)$\\

for fixed constants $\{\lambda_{i},\mu_{i}\}\subset\mathcal{R}$, $0\leq i\leq 5$. Let $r_{2}\in C^{\infty}_{0}([0,L]\times\mathcal{R})$ be given, as in Definition \ref{wave}, and $\phi_{0}(t)=r_{2}(t,0),\phi_{L}(t)=r_{2}(t,L)$, then, clearly, $\{\phi_{0},\phi_{L}\}\subset C^{\infty}(\mathcal{R})$, so clearly, we have that;\\

$p_{1,t}=\sum_{i=0}^{5}(\lambda_{i}\phi_{0}(t)+\mu_{i}\phi_{L}(t))x^{i}$\\

and $p_{1,t}\in C^{\infty}_{0}([0,L]\times\mathcal{R})$. Letting $F_{1}=p_{1,t}$, and $F_{2}=F-F_{1}$, we obtain the result.

\end{proof}

\begin{lemma}
\label{decomp2'}
Let $F\in C^{\infty}_{0}([0,L]\times\mathcal{R})$,  then there exists $\{F_{1},F_{2}\}\subset C^{\infty}_{0}([0,L]\times\mathcal{R})$, with $F^{(2j-1)}_{1,t,+}(0)=F^{(2j-1)}_{1,t,-}(L)=0$, $F^{(2j)}_{2,t,+}(0)=F^{(2j)}_{2,t,-}(L)=0$, for $1\leq j\leq n$, such that $F=F_{1}+F_{2}$.
\end{lemma}
\begin{proof}
This is just a uniform version of Lemma \ref{decomp'}. Let $v_{1,t}$ be defined as in Lemma \ref{decomp'}, replacing $\{f_{+}^{(2k)}(0),f_{-}^{(2k)}(L):1\leq k\leq n\}$ by $\{F_{t,+}^{(2k)}(0),F_{t,-}^{(2k)}(L):1\leq k\leq n\}$, and, let $p_{1,t}=T^{-1}(v_{1,t})$.\\

Then;\\

$p_{1,t}=\sum_{i=0}^{4n+1}d_{i}(t)x^{i}$\\

where the coefficients $d_{i}(t)=\sum_{k=1}^{n}(\lambda_{ik}F^{(2k)}_{t,+}(0)+\mu_{ik}F^{(2k)}_{t,-}(L)$\\

for fixed constants $\{\lambda_{ik},\mu_{ik}:0\leq i\leq 4n+1,1\leq k\leq n\}\subset\mathcal{R}$. Let $\{r_{2k}:1\leq k\leq n\}\subset C^{\infty}_{0}([0,L]\times\mathcal{R})$ be given, as in Definition \ref{wave}, and $\phi_{0,k}(t)=r_{2k}(t,0),\phi_{L,k}(t)=r_{2k}(t,L)$, then, $\{\phi_{0,k},\phi_{L,k}:1\leq k\leq n\}\subset C^{\infty}(\mathcal{R})$, (\footnote{We have that;\\

$lim_{h\rightarrow 0}({r_{2k}(L,t+h)-r_{2k}(L,t)\over h})=lim_{h\rightarrow 0}(lim_{x\rightarrow L}({r_{2k}(x,t+h)-r_{2k}(x,t)\over h}))$, $(*)$\\

As $r_{2k}\in C([-L,L]\times\mathcal{R})$, for fixed $h\neq 0$;\\

 $lim_{x\rightarrow L}{r_{2k}(x,t+h)-r_{2k}(x,t)\over h}={r_{2k}(L,t+h)-r_{2k}(L,t)\over h}$\\

For fixed $x'\neq L$;\\

$lim_{h\rightarrow 0}{r_{2k}(x',t+h)-r_{2k}(x',t)\over h}={\partial^{2k+1}F\over \partial x^{2k+1}}(x',t)$\\

and, moreover, the convergence is uniform for $x'\in (0,L)$, as ${\partial^{2k+1}F\over \partial x^{2k+1}}$ is bounded on $(0,L)\times (t-\epsilon,t+\epsilon)$, for any $\epsilon>0$. It follows that we can interchange the limits in $(*)$, to obtain that;\\

$lim_{h\rightarrow 0}({r_{2k}(L,t+h)-r_{2k}(L,t)\over h})$\\

$=lim_{x\rightarrow L}(lim_{h\rightarrow 0}({r_{2k}(x,t+h)-r_{2k}(x,t)\over h}))$\\

$=lim_{x'\rightarrow L}{\partial^{2k+1}F\over \partial x^{2k+1}}(x',t)=r_{2k+1}(L,t)$\\}). We have that;\\

$p_{1,t}=\sum_{i=0}^{4n+1}(\sum_{k=1}^{n}(\lambda_{ik}\phi_{0,k}(t)+\mu_{ik}\phi_{L,k})x^{i}$\\

and $p_{1,t}\in C^{\infty}_{0}([0,L]\times\mathcal{R})$. Letting $F_{1}=p_{1,t}$, and $F_{2}=F-F_{1}$, we obtain the result.

\end{proof}

\begin{lemma}
\label{extension2}
Let $F\in C^{\infty}_{0}([0,L]\times\mathcal{R})$, then, there exist $\{G_{1},G_{2},G\}\subset C^{2}([-L,L]\times\mathcal{R})$, such that, for all $\epsilon>0$ ;\\

$(i)$.$G|_{[\epsilon,L-\epsilon)\times\mathcal{R}}=F|_{[\epsilon,L-\epsilon)\times\mathcal{R}}$.\\

$(ii)$. $G_{1}$ is asymmetric and $\partial G_{1}\over\partial x$ is symmetric about $0$.\\

$(iii)$. $G_{2}$ is symmetric and $\partial G_{2}\over\partial x$ is asymmetric about $0$.\\

\end{lemma}

\begin{proof}
By Lemma \ref{decomp2}, we can find $\{F_{1},F_{2}\}\subset C^{\infty}_{0}([0,L]\times\mathcal{R})$, with $F'_{1,+}(0)=F'_{1,+}(L)=0$, $F''_{2,+}(0)=F''_{2,+}(L)=0$, such that $F=F_{1}+F_{2}$. By Lemma \ref{asymmetric2}, we can find $\{G_{1},G_{2}\}\subset C^{2}_{0}([-L,L])$, with $G_{1}|_{[0,L]}=G_{1}$, $G_{1}$ asymmetric and $\partial G_{1}\over\partial x$ symmetric about $0$, and with $G_{2}|_{[0,L]}=G_{2}$, $G_{2}$ symmetric and $\partial G_{2}\over\partial x$ asymmetric about $0$. Let $G=G_{1}+G_{2}$, then $G\in C^{2}_{0}([-L,L]\times\mathcal{R})$, and $G|_{[\epsilon,L-\epsilon)\times\mathcal{R}}=F|_{[\epsilon,L-\epsilon)\times\mathcal{R}}$, as required.
\end{proof}

\begin{lemma}
\label{extension2'}
Let $F\in C^{\infty}_{0}([0,L]\times\mathcal{R})$, then, there exist $\{G_{1},G_{2},G\}\subset C^{4}([-L,L]\times\mathcal{R})$, such that, for all $0\leq\epsilon<{L\over 2}$;\\

$(i)$.$G|_{[\epsilon,L-\epsilon)\times\mathcal{R}}=F|_{[\epsilon,L-\epsilon)\times\mathcal{R}}$.\\

$(ii)$. $G_{1},{\partial^{2}G_{1}\over\partial x^{2}}$ are asymmetric and ${\partial G_{1}\over\partial x},{\partial^{3} G_{1}\over\partial x^{3}}$ are symmetric about $0$.\\

$(iii)$. $G_{2},{\partial^{2} G_{2}\over\partial x^{2}}$ are symmetric and ${\partial G_{2}\over\partial x},{\partial^{3} G_{2}\over\partial x^{3}}$ are asymmetric about $0$.\\

\end{lemma}

\begin{proof}
By Lemma \ref{decomp2'}, we can find $\{F_{1},F_{2}\}\subset C^{\infty}_{0}([0,L]\times\mathcal{R})$, with $F^{(1)}_{1,+}(0)=F^{(1)}_{1,-}(L)=0$, $F^{(3)}_{1,+}(0)=F^{(3)}_{1,-}(L)=0$, $F^{(2)}_{2,+}(0)=F^{(2)}_{2,-}(L)=0$, $F^{(4)}_{2,+}(0)=F^{(4)}_{2,-}(L)=0$, such that $F=F_{1}+F_{2}$. By Lemma \ref{asymmetric2'}, we can find $\{G_{1},G_{2}\}\subset C^{4}_{0}([-L,L])$, with $G_{1}|_{[0,L]}=F_{1}$, $G_{1},{\partial^{2}G_{1}\over \partial x^{2}}$ asymmetric and ${\partial G_{1}\over\partial x}, {\partial^{3}G_{1}\over\partial x^{3}}$ symmetric about $0$, $G_{2}|_{[0,L]}=F_{2}$, $G_{2},{\partial^{2}G_{2}\over \partial x^{2}}$ symmetric and ${\partial G_{2}\over\partial x}, {\partial^{3}G_{2}\over\partial x^{3}}$ asymmetric about $0$  Let $G=G_{1}+G_{2}$, then $G\in C^{4}_{0}([-L,L]\times\mathcal{R})$, and $G|_{[\epsilon,L-\epsilon)\times\mathcal{R}}=F|_{[\epsilon,L-\epsilon)\times\mathcal{R}}$, as required.
\end{proof}

\begin{lemma}
\label{limit}
Let $F\in C^{\infty}_{0}([0,L]\times\mathcal{R})$ be a solution to the wave equation, then, for all $t\in\mathcal{R}$\\

$lim_{\epsilon\rightarrow 0}{\partial^{2}F\over \partial x^{2}}|_{(\epsilon,t)}=0$\\

$lim_{\epsilon\rightarrow 0}{\partial^{2}F\over \partial x^{2}}(L-\epsilon,t)=0$\\

\end{lemma}

\begin{proof}
Let $\{G_{1},G_{2},G\}$ be given as in Lemma \ref{extension2'}. Then, for all $t\in\mathcal{R}$, $G_{t}\in C^{4}([-L,L])$, and, using \cite{dep1}, the Fourier series expansion $\sum_{m\in\mathcal{Z}}c_{m}(t)e^{{\pi ixm\over L}}$ of $G_{t}$ converges uniformly to $G_{t}$ on $[-L,L]$, (\footnote{In fact, we only require that $G_{t}\in C^{2}([-L,L])$, see also \cite{SS}}). Similarly, as $G^{(n)}_{t}\in C^{2}([-L,L])$, for $0\leq n\leq 2$, the Fourier series expansion   $\sum_{m\in\mathcal{Z}}c_{m}(t)({\pi i m\over L})^{n}e^{{\pi ixm\over L}}$ of $G^{(n)}_{t}$, converges uniformly to $G^{(n)}_{t}$ on $[-L,L]$, for $0\leq n\leq 2$, $(*)$. We have that;\\

$c_{m}(t)={1\over 2L}\int_{-L}^{L}G(x,t)e^{-{\pi ixm\over L}}dx$\\

Hence, as, for $0\leq n\leq 4$, $t_{0}\in\mathcal{R}$, ${\partial^{n}G\over \partial x^{n}}$ is bounded on $[-L,L]\times (t_{0}-\delta,t_{0}+\delta)$, by the DCT, we have that $c_{\epsilon,m}\in C^{4}(\mathcal{R})$. Moreover, we have, for $0\leq n\leq 4$;\\

$c_{m}^{(n)}(t)={1\over 2L}\int_{-L}^{L}{\partial^{n}G_{\epsilon}\over \partial t^{n}}(x,t)e^{-{\pi ixm\over L}}dx$\\

Hence, again, as ${\partial^{n}G_{\epsilon,t}\over \partial t^{n}}\in C^{2}([-L,L])$, for $0\leq n\leq 2$, the Fourier series expansion $\sum_{m\in\mathcal{Z}}c_{m}^{(n)}(t)e^{{\pi ixm\over L}}$ of ${\partial^{n}G_{t}\over \partial t^{n}}$ converges uniformly to ${\partial^{n}G_{t}\over \partial t^{n}}$ on $[-L,L]$, for $0\leq n\leq 2$ . Then;\\

${\partial^{2}G_{t}\over \partial t^{2}}=\sum_{m\in\mathcal{Z}}c''_{m}(t)e^{{\pi ixm\over L}}$\\

${\partial^{2}G_{t}\over \partial x^{2}}=\sum_{m\in\mathcal{Z}}c_{m}(t)({\pi im\over L})^{2}e^{{\pi ixm\over L}}=-\sum_{m\in\mathcal{Z}}c_{m}(t)({\pi^{2}m^{2}\over L^{2}})e^{{\pi ixm\over L}}$\\

Using the facts that ${\partial^{2}G_{t}\over \partial t^{2}}={T\over \mu}{\partial^{2}G_{t}\over \partial x^{2}}$, on $(0,L)$, the series $\sum_{m\in\mathcal{Z}}[c''_{m}(t)+c_{m}(t)({\pi^{2}m^{2}T\over \mu L^{2}})]e^{{\pi ixm\over L}}$ is analytic on $[-L,L]$, and $\{e^{{\pi ixm\over L}}:m\in\mathcal{Z}\}$ are orthogonal on $[-L,L]$, we obtain that;\\

$c''_{m}(t)+c_{m}(t)({\pi^{2}m^{2}T\over \mu L^{2}})=0$ $(t\in\mathcal{R})$\\

$c_{m}(t)=A_{m}e^{i\pi m\sqrt{T}t\over L\sqrt{\mu}}+B_{m}e^{-{i\pi m\sqrt{T}t\over L\sqrt{\mu}}}$\\

with $\{A_{m},B_{m}\}\subset\mathcal{C}$, $A_{m}=a_{m}+ia'_{m}$, $B_{m}=b_{m}+ib'_{m}$ and;\\

$G=\sum_{m\in\mathcal{Z}}A_{m}e^{i\pi m\sqrt{T}t\over L\sqrt{\mu}}e^{{\pi ixm\over L}}+\sum_{m\in\mathcal{Z}}B_{m}e^{-{i\pi m\sqrt{T}t\over L\sqrt{\mu}}}e^{{\pi ixm\over L}}$\\

Then ${\partial^{2}G\over\partial x^{2}}=-[\sum_{m\in\mathcal{Z}}A_{m}{\pi^{2}m^{2}\over L^{2}}e^{i\pi m\sqrt{T}t\over L\sqrt{\mu}}e^{{\pi ixm\over L}}+\sum_{m\in\mathcal{Z}}B_{m}{\pi^{2}m^{2}\over L^{2}}e^{-{i\pi m\sqrt{T}t\over L\sqrt{\mu}}}e^{{\pi ixm\over L}}]$\\

$=-\sum_{m\in\mathcal{Z}_{\neq 0}}(a_{m}+b_{m}){\pi^{2}m^{2}\over L^{2}}cos({\pi xm\over L})cos({\pi m\sqrt{T}t\over L\sqrt{\mu}})+\theta(x,t)=S_{t}$\\

where $\theta(0,0)=\theta(L,0)=0$\\

We have that;\\

 $|(a_{m}+b_{m})|={1\over 2L}|\int_{-L}^{L}G_{0}(x)cos({\pi xm\over L})dx|$\\

 $\leq {L^{n-1}\over 2\pi^{n}m^{n}}\int_{-L}^{L}|G_{0}^{(n)}|dx\leq {C_{0,n}\over m^{n}}$, for $0\leq n\leq 4$\\

 where $C_{0,n}={L^{n-1}||G_{0}^{(n)}||_{L^{1}(-L,L)}\over 2\pi^{n}}$.\\

Then;\\

$|{\partial^{2}G_{0}\over\partial x^{2}}|(0)=S_{0}\leq \sum_{m\in\mathcal{Z}_{\neq 0}}|a_{m}+b_{m}|({\pi^{2}m^{2}\over L^{2}})$\\

$\leq \sum_{1\leq |m|\leq k-1}|a_{m}+b_{m}|({\pi^{2}m^{2}\over L^{2}})+\sum_{|m|\geq k}{C_{0,n}\over m^{n}}({\pi^{2}m^{2}\over L^{2}})$\\

$=\sum_{1\leq |m|\leq k-1}|a_{m}+b_{m}|(({\pi^{2}m^{2}\over L^{2}})+\sum_{|m|\geq k}{L^{n-3}||G_{0}^{(n)}||_{L^{1}(-L,L)}\over 2\pi^{n-2}m^{n-2}}$\\

Taking $n=4$, we obtain;\\

$S_{0}\leq \sum_{1\leq |m|\leq k-1}|a_{m}+b_{m}|({\pi^{2}m^{2}\over L^{2}})+\sum_{|m|\geq k}{L\over 2\pi^{2}m^{2}}||G_{0}^{(4)}||_{L^{1}(-L,L)}$\\

$\leq \sum_{1\leq |m|\leq k-1}|a_{m}+b_{m}|({\pi^{2}m^{2}\over L^{2}})+{L\over{2\pi^{2}(k-1)}}||G_{0}^{(4)}||_{L^{1}(-L,L)}$, $(**)$\\

We have, by conditions $(i),(ii)$ of Lemma \ref{extension2'} and the FTC, that, for all $0<\epsilon<L$;\\

$|a_{m}+b_{m}|\leq {1\over L}\int_{-\epsilon}^{\epsilon}|G_{0}(x)|dx+\int_{L-\epsilon}^{-L+\epsilon}|G_{0}(x)|dx$\\

$\leq {1\over L}(|G_{0}(\epsilon)|+|G_{0}(-\epsilon)|+|G_{0}(L-\epsilon)|+|G_{0}(-L+\epsilon)|)$\\

$={1\over L}(|F_{0}(\epsilon)|+|G_{1,0}(-\epsilon)|+|G_{2,0}(-\epsilon)|+|F_{0}(L-\epsilon)|+|G_{1,0}(-L+\epsilon)|$\\

\indent \ \ \ $+|G_{2,0}(-L+\epsilon)|)$\\

$\leq {2L^{2}\over \pi^{2}(k-1)}({\delta'\over 2})$\\

for sufficiently small $\epsilon(k,\delta')$, as $F_{0}\in C_{0}([0,L])$ and $\{G_{1,0},G_{2,0}\}\subset C_{0}([-L,L])$.  Taking $k\geq {4L||G_{0}^{(4)}||_{L^{1}(-L,L)}+1\over \pi^{2}\delta'}$, we then have that $|S_{0}|<\delta'$.  Then, using condition $(i)$ of Lemma \ref{extension2'}, and the fact that ${\partial^{2}G_{0}\over \partial x^{2}}$ is continuous at $0$, we obtain that $lim_{\epsilon\rightarrow 0}{\partial^{2}F_{0}\over \partial x^{2}}(\epsilon)=0$ as required. In a similar way, using an expansion around an arbitrary $t_{0}\in\mathcal{R}$, we obtain that $lim_{\epsilon\rightarrow 0}{\partial^{2}F_{t_{0}}\over \partial x^{2}}(\epsilon)=0$, as required. By exactly the same method, we obtain that $lim_{\epsilon\rightarrow 0}{\partial^{2}F_{t_{0}}\over \partial x^{2}}(L-\epsilon)=0$.
\end{proof}

\begin{lemma}
Let $F\in C^{\infty}_{0}([0,L]\times\mathcal{R}$ be a solution to the wave equation. Then, the Fourier series expansion of $F$ is given by;\\

$\sum_{m\in\mathcal{Z}_{>0}}K_{m}cos({\pi m\sqrt{T}t\over L\sqrt{\mu}})sin({\pi xm\over L})+L_{m}sin({\pi m\sqrt{T}t\over L\sqrt{\mu}})sin({\pi xm\over L})$\\

which converges uniformly to $F$ on $[0,L]$.

\end{lemma}

\begin{proof}

By Lemma \ref{limit}, we have that, for $t\in\mathcal{R}$, $lim_{\epsilon\rightarrow 0}{\partial^{2}F_{t}\over x^{2}}(\epsilon)=lim_{\epsilon\rightarrow 0}{\partial^{2}F_{t}\over x^{2}}(L-\epsilon)=0$. Using the fact;\\

$lim_{\epsilon\rightarrow 0}{\partial^{2}G_{1,t}\over x^{2}}(\epsilon)=lim_{\epsilon\rightarrow 0}{\partial^{2}G_{1,t}\over x^{2}}(L-\epsilon)=0$\\

from Lemma \ref{extension2'}, we obtain that;\\

$lim_{\epsilon\rightarrow 0}{\partial^{2}G_{2,t}\over x^{2}}(\epsilon)=lim_{\epsilon\rightarrow 0}{\partial^{2}G_{2,t}\over x^{2}}(L-\epsilon)=0$\\

Using Lemma 0.3 of \cite{dep}, we obtain that;\\

$lim_{\epsilon\rightarrow 0}{\partial^{4}G_{2,t}\over x^{4}}(\epsilon)=lim_{\epsilon\rightarrow 0}{\partial^{4}G_{2,t}\over x^{4}}(L-\epsilon)=0$\\

Hence, by Definition of $G_{2}$ in \ref{decomp2'},\ref{extension2'}, we obtain that $G_{2}=0$. It follows that there exists $G_{1}\in C^{4}_{0}([-L,L]\times\mathcal{R})$, with $G_{1}$ asymmetric about $0$, such that $G_{1}|_{[0,L]}=F$.\\

Let $h\in C^{4}_{0}([-L,L])$ be an asymmetric function, and let;\\

$h(x)=\sum_{m\in\mathcal{Z}}\hat{h}(m)e^{\pi ixm\over L}$ be the Fourier series expansion of $h$, with;\\

$\hat{h}(m)={1\over 2L}\int_{-L}^{L}h(x)e^{-\pi ixm\over L}$, for $m\in\mathcal{Z}$\\

We have that;\\

$\hat{h}(m)={1\over 2L}\int_{-L}^{L}h(x)cos({\pi xm\over L})dx-{i\over 2L}\int_{-L}^{L}h(x)sin({\pi xm\over L})dx$\\

$={-i\over 2L}\int_{-L}^{L}h(x)sin({\pi xm\over L})dx={-i\over L}\int_{0}^{L}f(x)sin({\pi xm\over L})=ie_{m}$\\

with $e_{m}=-e_{-m}$, for $m\geq 0$, so $e_{0}=0$. Then;\\

$h(x)=-\sum_{m\in\mathcal{Z}_{>0}}2e_{m}sin({\pi xm\over L})$ \\

Then writing;\\

$G_{1}(t,x)=\sum_{m\in\mathcal Z_{>0}}f_{m}(t)sin({\pi xm\over L})$\\

and substituting into $(*)$ of Definition \ref{wave}, justified by the method of Lemma \ref{limit} and the fact that $G_{1}\in C^{4}([-L,L]\times\mathcal{R})$, we have that;\\

$\sum_{m\in\mathcal{Z}_{>0}}f_{m}''(t)sin({\pi xm\over L})=-{T\over \mu}(\sum_{m\in\mathcal{Z}}f_{m}(t)({\pi m\over L})^{2} sin({\pi xm\over L}))$\\

Hence, $f_{m}''(t)=-{T\over \mu}f_{m}(t)({\pi m\over L}^{2})=-{\pi^{2}m^{2}T\over L^{2}\mu}f_{m}(t)$\\

$f_{m}(t)=K_{m}cos({\pi m\sqrt{T}t\over L\sqrt{\mu}})+L_{m}sin({\pi m\sqrt{T}t\over L\sqrt{\mu}})$\\

giving;\\

$G_{1}(t,x)=\sum_{m\in\mathcal Z_{>0}}K_{m}cos({\pi m\sqrt{T}t\over L\sqrt{\mu}})sin({\pi xm\over L})+L_{m}sin({\pi m\sqrt{T}t\over L\sqrt{\mu}})sin({\pi xm\over L})$\\

where the convergence is uniform on $[-L,L]$. Using the fact that $G_{1}|_{[0,L]}=F$, by Lemma \ref{decomp2}, we obtain that the series converges uniformly to $F$ on $[0,L]$ as required.

\end{proof}

\end{document}